\documentclass[12pt,reqno]{amsart}

\usepackage{amsmath,amsfonts,amssymb,amsthm,bm,color}

\usepackage{geometry}
\usepackage{hyperref}
\hypersetup{colorlinks,linkcolor=[rgb]{0,0,1},citecolor=[rgb]{1,0,0}}

\usepackage[all,cmtip]{xy}

\usepackage[utf8]{inputenc}

\usepackage{multirow}

\usepackage{mathrsfs}

\newlength{\wideitemsep}
\let\olditem\item
\renewcommand{\item}{\setlength{\itemsep}{\wideitemsep}\olditem}

\renewcommand{\geq}{\geqslant}

\newcommand{\lra}{\longrightarrow}

\newcommand{\into}{\hookrightarrow}

\newtheorem{theorem}{Theorem}[section]
\newtheorem{lemma}[theorem]{Lemma}
\newtheorem{proposition}[theorem]{Proposition}
\newtheorem{corollary}[theorem]{Corollary}
\newtheorem{conjecture}[theorem]{Conjecture}

\newtheorem*{conjecture*}{Conjecture}

\theoremstyle{definition}
\newtheorem{definition}[theorem]{Definition}

\theoremstyle{remark}
\newtheorem{remark}[theorem]{Remark}

\newcommand{\cA}{\mathscr{A}}
\newcommand{\cB}{\mathscr{B}}
\newcommand{\cC}{\mathscr{C}}
\newcommand{\cD}{\mathscr{D}}

\newcommand{\cF}{\mathcal{F}}

\newcommand{\cI}{\mathcal{I}}

\newcommand{\cM}{\mathcal{M}}

\newcommand{\cT}{\mathscr{T}}

\newcommand{\kk}{\ensuremath{\Bbbk}} 

\newcommand{\CC}{\ensuremath{\mathbb{C}}}

\newcommand{\QQ}{\ensuremath{\mathbb{Q}}}

\newcommand{\ZZ}{\ensuremath{\mathbb{Z}}} 

\newcommand{\vv}{{\bm{\mathrm v}}}  
\newcommand{\ww}{{\bm{\mathrm w}}}
\newcommand{\uu}{{\bm{\mathrm u}}}  

\DeclareMathOperator{\Amp}{Amp}

\DeclareMathOperator{\ch}{ch}

\DeclareMathOperator{\Coh}{Coh}

\DeclareMathOperator{\cone}{Cone}

\DeclareMathOperator{\Ext}{Ext}

\DeclareMathOperator{\Hom}{Hom}

\DeclareMathOperator{\NS}{NS}
\DeclareMathOperator{\Ob}{Ob}

\DeclareMathOperator{\RHom}{RHom}
\DeclareMathOperator{\rk}{rk}

\DeclareMathOperator{\Stab}{Stab}

\DeclareMathOperator{\td}{td}

\DeclareMathOperator{\dgCat}{\mathbf{dgCat}}
\DeclareMathOperator{\Hqe}{\mathbf{Hqe}}

\newcommand{\hdot}{{\:\raisebox{3pt}{\text{\circle*{1.5}}}}}

\begin{document}

\title{Formality conjecture for K3 surfaces}

\author{Nero Budur}
\address{KU Leuven\\
Celestijnenlaan 200B\\
B-3001 Leuven\\
Belgium}
\email{nero.budur@kuleuven.be}
\urladdr{https://www.kuleuven.be/wis/algebra/budur}

\author{Ziyu Zhang}
\address{Institut f\"ur algebraische Geometrie\\
Leibniz Universit\"at Hannover\\
Welfengarten 1\\
30167 Hannover\\
Germany}
\email{zhangzy@math.uni-hannover.de}
\urladdr{https://ziyuzhang.github.io}



\begin{abstract}
	We give a proof of the formality conjecture of Kaledin and Lehn: on a complex projective K3 surface, the DG algebra $\RHom^\hdot(F,F)$ is formal for any sheaf $F$ polystable with respect to an ample line bundle. Our main tool is the uniqueness of the DG enhancement of the bounded derived category of coherent sheaves. We also extend the formality result to derived objects that are polystable with respect to a generic Bridgeland stability condition.
\end{abstract}

\maketitle

\section{Introduction}

A differential graded algebra is said to be {formal} if it is quasi-isomorphic to its homology algebra.
Formality is a central concept with roots in real homotopy theory \cite{DGMS75}, deformation quantization \cite{Kon03}, and deformation theory \cite{GM88}. The presence of formality is in general difficult to prove and has important consequences: in topology, formality implies vanishing of all Massey products; in deformation quantization, it implies that every Poisson manifold has a deformation quantization; in deformation theory, it implies that the underlying derived, or non-commutatively thickened, moduli spaces are locally defined by cup products, leading to a linear-algebraic (quiver) interpretation.  The following was conjectured by Kaledin and Lehn in \cite{KL07}: 

\begin{conjecture}
	\label{conj:conjB} 
	For a projective K3 surface $X$ with a generic polarization $H$, let $F$ be a $H$-polystable coherent sheaf on $X$. Then the DG algebra $\RHom^\hdot(F,F)$ is formal.
\end{conjecture}

In this article we prove two generalizations of this conjecture:

\begin{theorem}
	\label{thm:intro}
	For a projective K3 surface $X$ with an arbitrary polarization $H$, let $F$ be a $H$-polystable coherent sheaf on $X$. Then the DG algebra $\RHom^\hdot(F,F)$ is formal.
\end{theorem}

\begin{theorem}
	\label{cor:intro}
	Let $X$ be a projective K3 surface, $\vv$ a Mukai vector, and $\sigma \in \Stab^\dagger(X)$ a Bridgeland stability condition that is generic with respect to $\vv$. For any $\sigma$-polystable derived object $F$ of Mukai vector $\vv$, the DG algebra $\RHom^\hdot(F,F)$ is formal.
\end{theorem}

Conjecture \ref{conj:conjB} was proved in  \cite{KL07} for $F = \cI_Z^{\oplus n}$, where $\cI_Z$ is the ideal sheaf of a $0$-dimensional closed subscheme. In \cite{Zha12}, the conjecture was proved in a few more  cases; see Proposition \ref{prop:known}. It was also pointed out in \cite{Zha12} that the technique from \cite{KL07} is no longer enough for tackling the remaining cases. In this article we explore the following new idea:

\begin{proposition}
	\label{prop:intro}
	For smooth projective varieties $X$ and $Y$, let
	\begin{equation*}
		\Phi: D^b(\Coh(X)) \longrightarrow D^b(\Coh(Y))
	\end{equation*}
	be a derived equivalence. Then for any object $F \in D^b(\Coh(X))$, the DG algebra $\RHom^\hdot(F,F)$ is formal if and only if $\RHom^\hdot(\Phi(F), \Phi(F))$ is formal.
\end{proposition}

The main ingredient in the proof of  is a theorem of Lunts and Orlov \cite[Theorem 2.14]{LO10} stating that for a smooth projective variety $X$, $D^b(\Coh(X))$ admits a strongly unique DG enhancement. The proof of Theorem \ref{thm:intro} follows  easily from Proposition \ref{prop:intro}: given any polystable sheaf $F$, there exists some Fourier-Mukai transform $\Phi$ by \cite[Theorem 1.7]{Yos09}, such that $\Phi(F)$ is a polystable sheaf and satisfies the assumption required in \cite{Zha12}. Similarly, Theorem \ref{cor:intro} is reduced to Theorem \ref{thm:intro} by applying autoequivalences of $D^b(\Coh(X))$ constructed using \cite{Bri08,Yos09,BM14a,BM14b}. 

For the relation between Conjecture \ref{conj:conjB} and  non-commutative deformation theory see for example \cite{Tod17a}; for the relation to derived deformation theory see \cite{Toe17}. Formality implies that all non-trivial analogs of the higher Massey products on $\Ext^\hdot(F,F)$ are zero. Thus our  results imply immediately the local quadraticity of the Kuranishi spaces of polystable sheaves on K3 surfaces obtained by \cite[Corollary 0.6]{Yos16}, and in a particular case by \cite[Theorem 3.7]{AS18}. Our approach of reducing everything to \cite{Zha12} is inspired by their approaches.

Our results have several other immediate implications. We can recover the main theorem of \cite[Theorem 1.1]{AS18} on symplectic resolutions of moduli spaces $\cM_{X,H}(\vv)$ of $H$-semistable sheaves on K3 surfaces via variations of GIT quotients of quiver varieties, by combining Theorem \ref{thm:intro} and  \cite[Theorem 1.3]{Tod17b}, as explained in \cite{Tod17b}. We can also obtain that $\cM_{X,H}(\vv)$ has symplectic singularities for arbitrary polarizations $H$, by combining Theorem \ref{thm:intro} and \cite[Proposition 1.2]{BS16}, which generalizes slightly \cite[Theorem 6.2]{KLS06}. Moreover, our results and the recent \cite{B-rat} imply immediately that the Kuranishi spaces of polystable sheaves have rational singularities, see \cite[Remarks 3.5 and 4.5]{B-rat}.

It is worth noting that Theorem \ref{thm:intro} fails if $F$ is not polystable. In \cite{LU18} several families of K3 surfaces $X$ are exhibited, on which the formality of $\RHom^\hdot(F,F)$ fails for a split generator $F$ of $D^b(\Coh(X))$.

Our results should have a symplectic counterpart. Kontsevich's Homological Mirror Symmetry Conjecture says that $D^b(\Coh(X))$ is equivalent to a triangulated category constructed from the symplectic geometry of the mirror of a Calabi-Yau manifold $X$. Our Corollary \ref{cor:qisom} below guarantees that formality would be mirrored. However, before any mirror conclusions can be drawn, a technical hurdle must be passed: we use $\CC$ coefficients, whereas HMS uses Novikov rings; see \cite{Sei}.
 
Throughout the article, the ground field is $\CC$, except in \S \ref{sec:category} and in Proposition \ref{prop:intro} where we can allow an arbitrary ground field $\kk$. In \S \ref{sec:category} we prove Proposition \ref{prop:intro}, while the proofs of Theorems \ref{thm:intro} and \ref{cor:intro} will be given in \S \ref{sec:geometry}. 

At the request of the referees, we have removed in this version some material known to experts. The considerably more detailed earlier version \cite{BZv3} of this article remains available on arXiv.

{\it Acknowledgements.} We thank T. Bridgeland for pointing to us the relevance of \cite{LO10} to the formality conjecture. We also thank E. Arbarello, A. Bayer, A. King, A. Krug, Y. Lekili, G. Sacc\`a and K. Yoshioka for their help with questions and for inspiring discussions. The second named author thanks the organizers of the \emph{School on Deformation Theory (University of Turin, July 2017)} for the stimulating atmosphere and the excellent lectures related to this topic. The first named author was partly supported by KU Leuven OT, FWO, and Methusalem grants.

\section{Formality via Uniqueness of DG Enhancement}\label{sec:category}

\subsection{Generalities on DG categories}\label{sec:DGgen}

We collect some classical concepts following mainly \cite{LO10}. We  work over a fixed field $\kk$. All categories are assumed to be small and $\kk$-linear.

\begin{definition}
A \emph{DG category} is a $\kk$-linear category $\cA$ whose morphism spaces $\Hom(A_1,A_2)$ are DG $\kk$-modules (aka complexes of $\kk$-vector spaces), such that
\begin{equation}
\label{eqn:comp}
\Hom_\cA(A_1,A_2) \otimes \Hom_\cA(A_2,A_3) \lra \Hom_\cA(A_1,A_3)
\end{equation}
are morphisms of DG $\kk$-modules for any objects $A_1, A_2, A_3 \in \Ob(\cA)$. Moreover, for any $A \in \Ob(\cA)$, there is an identity morphism $1_A \in \Hom_\cA(A,A)$ which is closed of degree $0$ and compatible with the composition.
\end{definition}

\begin{remark}
	The definition implies that the graded Leibniz rule holds and $\Hom_\cA(A,A)$ is a DG algebra for any $A \in \Ob(\cA)$.
\end{remark}

\begin{definition}
The \emph{homotopy category} $H^0(\cA)$ of a DG category $\cA$ is a $\kk$-linear category with the same objects as in $\cA$ and morphism spaces
$$ \Hom_{H^0(\cA)}(A_1, A_2) = H^0(\Hom_\cA(A_1,A_2)) $$
for any $A_1, A_2 \in \Ob(\cA)$.
\end{definition}

\begin{definition}
A \emph{DG functor} $\cF: \cA \to \cB$ between two DG categories is given by a map of sets
$$ \cF: \Ob(\cA) \to \Ob(\cB) $$
and morphisms of DG $\kk$-modules
$$ \cF(A_1,A_2): \Hom_\cA(A_1,A_2) \lra \Hom_\cB(\cF(A_1), \cF(A_2)) $$
for any $A_1, A_2 \in \Ob(\cA)$, compatible with compositions \eqref{eqn:comp} and units.
\end{definition}

\begin{definition}
A DG functor $\cF: \cA \to \cB$ is called a \emph{quasi-equivalence} if $\cF(X,Y)$ is a quasi-isomorphism for any $X, Y \in \Ob(\cA)$ and the induced functor on the homotopy categories
$$ H^0(\cF): H^0(\cA) \lra H^0(\cB) $$
is an equivalence of categories.
\end{definition}

\begin{remark}
In fact, instead of requiring $H^0(\cF)$ to be an equivalence, it is sufficient to require it to be essentially surjective, as the full faithfulness is already encoded in the quasi-isomorphisms of morphism spaces. See e.g. \cite[Definition 2, \S 2.3]{Toe11}.
\end{remark}

We denote the category of small DG categories with DG functors as morphisms by $\dgCat$, and its localization with respect to quasi-equivalences by $\Hqe$. It was proven in \cite{Tab05} that $\dgCat$ has the structure of a model category, with quasi-equivalences being the weak equivalences in the model structure. Then $\Hqe$ is the homotopy category of this model category. One special property of this model structure on $\dgCat$ is that every small DG category is a fibrant object.

\begin{definition}
	A morphism between two DG categories in $\Hqe$ is called a \emph{quasi-functor}. We say two DG categories are \emph{quasi-equivalent} if they are isomorphic in $\Hqe$.
\end{definition}

By this definition, two quasi-equivalent DG categories can be connected by a zig-zag chain of DG functors with alternative arrow directions. In fact, one has the following simpler presentation for a quasi-functor (see \cite[p.858]{LO10}; we supply a proof for the sake of completeness):

\begin{lemma}
	\label{lem:oneroof}
	Let $\cA$ and $\cB$ be DG categories. Any quasi-functor from $\cA$ to $\cB$ can be represented by the diagram
	\begin{equation}
		\label{eqn:factor}
		\cA \stackrel{f}{\longleftarrow} \cC \stackrel{g}{\longrightarrow} \cB
	\end{equation}
	where $\cC$ is a DG category, $f$ and $g$ are DG functors, with $f$ being a quasi-equivalence. Moreover, $\cA$ and $\cB$ are quasi-equivalent if and only if $g$ is also a quasi-equivalence.
\end{lemma}

\begin{proof}	By the fundamental theorem of model categories \cite[Theorem 1.2.10 (ii)]{Hov99}, we can represent a quasi-functor from $\cA$ to $\cB$ by a DG functor $\cC \stackrel{g}{\longrightarrow} \cD$, where $\cC \stackrel{f}{\longrightarrow} \cA$ is a cofibrant replacement of $\cA$, and $\cB \stackrel{h}{\longrightarrow} \cD$ is a fibrant replacement of $\cB$. Since $\cB$ itself is a fibrant object, we can choose $\cD = \cB$ and $h$ the identity functor. Hence we get \eqref{eqn:factor}. The other statement follows from \cite[Theorem 1.2.10 (iv)]{Hov99}.\end{proof}

For any DG category $\cA$, it was constructed in \cite{BK91} the \emph{pre-triangulated hull} $\cA^{\text{pre-tr}}$ of $\cA$ by formally adding to $\cA$ all shifts, all cones of morphisms, and cones of morphisms between cones, etc. There is a canonical embedding of DG categories $\cA \hookrightarrow \cA^{\text{pre-tr}}$.

\begin{definition}
\label{def:pretri}
	A DG category $\cA$ is said to be \emph{pre-triangulated} if for every object $A \in \cA$ and $n \in \ZZ$, the object $A[n] \in \cA^{\text{pre-tr}}$ is homotopy equivalent to an object in $\cA$, and for every closed morphism $f$ in $\cA$ of degree $0$, the object $\cone(f) \in \cA^{\text{pre-tr}}$ is homotopy equivalent to an object in $\cA$. 
\end{definition}

\begin{remark}
	In other words, a DG category $\cA$ is pre-triangulated if and only if the DG functor $\cA \hookrightarrow \cA^{\text{pre-tr}}$ is a quasi-equivalence; equivalently, the embedding of the homotopy categories $H^0(\cA) \into H^0(\cA^{\textrm{pre-tr}})$ is an equivalence. In such a case, $H^0(\cA)$ is naturally a triangulated category.
\end{remark}

\begin{definition}
A \emph{DG enhancement} of a triangulated category $\cT$ is a pair $(\cB, e)$, where $\cB$ is a pre-triangulated DG category and $e: H^0(\cB) \to \cT$ is an equivalence of triangulated categories.
\end{definition}

\begin{definition}
\label{def:unienh}
We say a triangulated category $\cT$ has a \emph{unique DG enhancement} if, given two DG enhancements $(\cB, e)$ and $(\cB', e')$ of $\cT$, there exists a quasi-functor $\cF: \cB \to \cB'$ such that $H^0(\cF): H^0(\cB) \to H^0(\cB')$ is an equivalence of triangulated categories. We say $\cT$ has a \emph{strongly unique DG enhancement} if moreover $\cF$ can be chosen so that the functors $e$ and $e' \circ H^0(\cF)$ are isomorphic.
\end{definition}

\subsection{Preservation of formality}\label{sec:unique}

We explain now why the uniqueness of DG enhancement of a triangulated category  helps with formality problems. The key is the following result. Althought it might be known to experts, we nevertheless supply a proof since we do not know of any in the literature.

\begin{proposition}
	\label{prop:RHom}
	Let $(\cB_1, e_1)$ and $(\cB_2, e_2)$ be DG enhancements of a triangulated category $\cT$. Let $T \in \Ob(\cT)$, $B_1 \in \Ob(\cB_1)$, $B_2 \in \Ob(\cB_2)$, such that $T$, $e_1(B_1)$ and $e_2(B_2)$ are isomorphic in $\cT$. Assume that $\cT$ has a strongly unique DG enhancement, then $\Hom_{\cB_1}(B_1,B_1)$ and $\Hom_{\cB_2}(B_2,B_2)$ are quasi-isomorphic DG algebras.
\end{proposition}

\begin{proof}
	For $i=1$ and $2$, we construct a full subcategory $\cC_i$ of $\cB_i$, whose objects are given by
	$$ \Ob(\cC_i) = \Ob(\cB_i) \backslash \{ B \in \Ob(\cB_i) \mid e_i(B) \cong T \text{ in } \cT, B \neq B_i \}. $$
	Clearly $\cC_i$ is also a DG category. We claim it is pre-triangulated. Indeed, let $\cB_i^{\text{pre-tr}}$ and $\cC_i^{\text{pre-tr}}$ be the pre-triangulated hulls of $\cB_i$ and $\cC_i$ respectively. Then all functors in the commutative diagram
	\begin{equation*}
		\xymatrix{
		H^0(\cB_i) \ar@{^{(}->}[r] & H^0(\cB_i^{\text{pre-tr}}) \\
		H^0(\cC_i) \ar@{^{(}->}[r] \ar@{^{(}->}[u] & H^0(\cC_i^{\text{pre-tr}}) \ar@{^{(}->}[u]
		}
	\end{equation*}
	are fully faithful. By the assumption that $\cB_i$ is pre-triangulated, the upper horizontal arrow is an equivalence. By the construction of $\cC_i$, the left vertical arrow is also an equivalence. In particular, they are essentially surjective. Hence the bottom horizontal arrow must be essentially surjective, hence an equivalence, which proves that $\cC_i$ is a pre-triangulated DG category. Moreover, since the composition
	$$ H^0(\cC_i) \into H^0(\cB_i) \stackrel{e_i}{\rightarrow} \cT $$
	is an equivalence of categories, we conclude that $\cC_i$ is a DG enhancement of $\cT$.
	
	By assumption, $\cT$ has a strongly unique DG enhancement. Therefore by Lemma \ref{lem:oneroof}, there exists some DG category $\cC_0$, such that both functors $f_1$ and $f_2$ in the roof
	$$ \cC_1 \stackrel{f_1}{\longleftarrow} \cC_0 \stackrel{f_2}{\longrightarrow} \cC_2 $$
	are quasi-equivalences. In particular, all functors in the diagram
	\begin{equation*}
		\xymatrix{
		H^0(\cC_0) \ar[r]^{H^0(f_2)} \ar[d]_{H^0(f_1)} & H^0(\cC_2) \ar[d] \\
		H^0(\cC_1) \ar[r] & \cT
		}
	\end{equation*}
	are equivalences of categories, and the diagram is $2$-commutative (the two compositions from $H^0(\cC_0)$ to $\cT$ are isomorphic functors).

	By the essential surjectivity of $f_1$, there exists some $B_0 \in \Ob(\cC_0)$, such that $f_1(B_0) \cong B_1$ in $H^0(\cC_1)$. By the construction of $\cC_1$, we know that $B_1$ is the only object in its isomorphism class of objects in $H^0(\cC_1)$, hence $f_1(B_0)=B_1$. By the $2$-commutativity of the diagram, the images of $f_2(B_0)$ and $B_2$ are both isomorphic to $T$ in $\cT$, hence $f_2(B_0)$ and $B_2$ themselves are in the same isomorphism class of objects in $H^0(\cC_2)$, which implies $f_2(B_0)=B_2$ by the construction of the category $\cC_2$.
	
	Since $f_1$ and $f_2$ are quasi-equivalences, the morphism
	\begin{equation*}
		f_i(B_0,B_0): \Hom_{\cC_0}(B_0, B_0) \lra \Hom_{\cC_i}(B_i, B_i)
	\end{equation*}
	is a quasi-isomorphism of DG algebras for $i=1$ and $2$. Since $\cC_i$ is a full subcategory of $\cB_i$ for $i=1$ and $2$, we conclude that $\Hom_{\cB_1}(B_1, B_1)$ and $\Hom_{\cB_2}(B_2, B_2)$ are quasi-isomorphic DG algebras.
	\end{proof}

\begin{remark}
	\label{rmk:RHom}
	Under the assumption of the above proposition, we can associate canonically to any $T \in \Ob(\cT)$ a DG algebra $\Hom_{\cB_1}(B_1,B_1)$ (for any lift $B_1$ of $T$ in any DG enhancement $\cB_1$ of $\cT$), which is well-defined up to quasi-isomorphisms. For convenience, we will denote this (quasi-isomorphism class of) DG algebra by $\RHom^\hdot(T,T)$.
\end{remark}

The following alternative formulation of the proposition is useful:

\begin{corollary}
	\label{cor:qisom}
	Let $\Phi: \cT_1 \to \cT_2$ be an equivalence of triangulated categories. Assume that $\cT_2$ (hence $\cT_1$) has a strongly unique DG enhancement. Then for any object $T \in \Ob(\cT_1)$, the DG algebras $\RHom^\hdot(T,T)$ and $\RHom^\hdot(\Phi(T), \Phi(T))$ are quasi-isomorphic. In particular, $\RHom(T,T)$ is formal if and only if $\RHom(\Phi(T),\Phi(T))$ is formal. \qed
\end{corollary}

\noindent{\bf Proof of Proposition \ref{prop:intro}.}
	It follows immediately from Corollary \ref{cor:qisom} and \cite[Theorem 2.14]{LO10} which states that $D^b(\Coh(X))$ has a strongly unique DG enhancement for a smooth projective variety $X$. \qed

\section{Formality on K3 Surfaces}\label{sec:geometry}

\subsection{Formality for coherent sheaves}\label{sec:moduli}
From now,  $(X,H)$ is a complex projective K3 surface and $\kk = \CC$. Let $F$ be a coherent sheaf on $X$. The \emph{Mukai vector} of $F$ is
$$ \vv = \ch(F) \cdot \sqrt{\td(X)}\in H^0(X, \ZZ) \oplus \NS(X) \oplus H^4(X, \ZZ) = H^\ast_{alg}(X, \ZZ). $$
If we write $\vv = (\vv_0, \vv_1, \vv_2)$, then the \emph{dual} of $\vv$ is defined by
$ \vv^\vee = (\vv_0, -\vv_1, \vv_2). $
The \emph{Mukai pairing} on $H^\ast_{alg}(X, \ZZ)$  is defined by
$$ \vv \cdot \ww = -\vv_0\ww_2 + \vv_1\ww_1 - \vv_2\ww_0, $$
where the products on the right hand side are Poincar\'e pairings.

Recall from \cite{HL10} that for a coherent sheaf one has the notions of {\it $H$-(semi)stability} (Gieseker) and {\it $\mu_H$-(semi)stability} (slope). An $H$-semistable sheaf $F$ is $H$-\emph{polystable} if it can be written in the form of
\begin{equation}
	\label{eqn:sum}
	F = F_1^{\oplus n_1} \oplus \cdots \oplus F_k^{\oplus n_k}
\end{equation}
where $F_1, \cdots, F_k$ are pairwise non-isomorphic $H$-stable summands, and $n_1, \cdots, n_k$ are strictly positive integers.

The moduli space of $H$-semistable coherent sheaves on $X$ of Mukai vector $\vv$ is denoted by $\cM_{X,H}(\vv)$. The closed points of $\cM_{X,H}(\vv)$ are into one-to-one correspondence with the $H$-polystable sheaves, and with the $S$-equivalence classes of semistable sheaves.

The following  result generalized a special case proved in \cite[Proposition 3.1]{KL07}:

\begin{proposition}[{\cite[Proposition 1.3]{Zha12}}]
\label{prop:known}
Let $(X,H)$ be a  projective K3 surface, and $\vv$ a Mukai vector of positive rank. Assume $H$ is generic with respect to $\vv$, and there is at least one $\mu_H$-stable sheaf of Mukai vector $\vv$. Let $F$ be a $H$-polystable sheaf with a decomposition given by \eqref{eqn:sum}. Assume either
\begin{itemize}
\item[(i)] $\rk F_i=1$ for all $i=1, \cdots, k$; or
\item[(ii)] $\rk F_i\geq 2$ for all $i=1, \cdots, k$.
\end{itemize}
Then the DG algebra $\RHom^\hdot(F,F)$ is formal.
\end{proposition}

\begin{remark}
	\label{rmk:nongen}
	By \cite[Remark 3.4 (2)]{AS18}, that the assumption of $H$ being generic with respect to $\vv$ is not necessary for the case (ii). However, case (i) does require it.
\end{remark}

\begin{remark}
	\label{rmk:nomust}
	As explained in \cite[\S 2]{Zha12}, instead of requiring the existence of a $\mu_H$-stable sheaf of Mukai vector $\vv$, it suffices to require that each stable summand $F_i$ has a $\mu_H$-stable deformation in its own moduli.
\end{remark}

Define the integral functor 
\begin{align}
	\Phi: D^b(\Coh(X))\longrightarrow\ D^b(\Coh(X)) \label{eqn:FMtrans}\\
	F\longmapsto\ \mathbf{R}q_*(p^*F \stackrel{\mathbf{L}}{\otimes} \cI_\Delta), \notag
\end{align}
where $p$ and $q$ are the first and the second projection from $X\times X$, and $\cI_\Delta$ is the ideal sheaf of the diagonal embedding $X \hookrightarrow X \times X$. It is an autoequivalence, \cite[Examples 10.9]{Huy06}. 
The following result of Yoshioka is crucial:

\begin{theorem}[{\cite[Proposition 1.5, Theorem 1.7]{Yos09}}]
	\label{thm:preserve}
	Let $(X,H)$ be a projective K3 surface. Then for any $F \in D^b(\Coh(X))$ with Mukai vector $$ \vv = (r, dH+D, a) $$ for some $r,a \in \ZZ$, $d \in \QQ$ and $D \in \NS(X)_\QQ \cap H^\perp$, the Mukai vector of $\Phi(F)$ can be given by $$ \widehat{\vv} = (a, -(dH+\widehat{D}), r) $$ for some $\widehat{D} \in \NS(X)_\QQ \cap H^\perp$. Moreover, in either of the two following cases:
	\begin{enumerate}
		\item $r > 0$, $a > 0$, and $ d > \max \left\{ 4r^2 + {1}/{H^2},\ 2r(\vv^2-D^2) \right\}, $
		\item $r = 0$, and $ a > \max \left\{ 3,\ (\vv^2-D^2)/2 + 1 \right\}, $
	\end{enumerate}
	 $\Phi$  induces an isomorphism $\cM_{X,H}(\vv) \cong \cM_{X,H}(\widehat{\vv})$ preserving $H$-polystability, such that, for every $H$-polystable sheaf $F$, $\Phi(F_i)$ is $\mu_H$-stable for each $H$-stable summand $F_i$ of $F$.
\end{theorem}

In fact, by \cite[Examples 10.9, (ii)]{Huy06} one has $D=\widehat{D}$ in the above theorem.

\noindent
{\bf{Proof of Theorem \ref{thm:intro}.}}
	We write $\vv = (r, dH+D, a)$ for the Mukai vector of $F$ as above. Suppose  that $(r, dH+D) \neq (0,0)$. Then $r \geqslant 0$, and $r=0$ would imply $d>0$ since otherwise $dH+D$ is not effective. Then for any positive integer $m$,
	\begin{equation}
		\label{eqn:tensorH}
		\vv \cdot e^{mH} = \left( r, dH+D+rmH, a+dmH^2+\frac{1}{2}rm^2H^2 \right).
	\end{equation}
	Denote the right hand side of \eqref{eqn:tensorH} by $\vv'=(\vv'_0, \vv'_1, \vv'_2)$. For $m \gg 0$, the following two conditions hold:
	\begin{itemize}
		\item[($\dagger$)] The vector $\vv'$ satisfies (either of) the conditions in Theorem \ref{thm:preserve} (depending on whether the rank of $\vv$ is positive or zero); c.f. \cite[Remark 1.4]{Yos09}.
		\item[($\ddagger$)] Either $0 < \vv'_0 < \vv'_2$ or $0 < H \cdot \vv'_1 < \vv'_2$ (depending on whether the rank of $\vv$ is positive or zero).
	\end{itemize}

	Consider the composition of autoequivalences
	\begin{equation*}
		\xymatrix{
		D^b(\Coh(X)) \ar[r]^{- \otimes H^m} & D^b(\Coh(X)) \ar[r]^{\Phi} & D^b(\Coh(X)).
		}
	\end{equation*}
	 For an  $H$-polystable sheaf $F$ with a decomposition \eqref{eqn:sum}, the Mukai vector of $F \otimes H^m$ is $\vv'=\vv \cdot e^{mH}$. The condition $(\dagger)$ guarantees that $\Phi(F \otimes H^m)$ is an $H$-polystable sheaf by Theorem \ref{thm:preserve}, which can be decomposed into stable summands in the form of
	$$ \Phi(F \otimes H^m) = \Phi(F_1 \otimes H^m)^{\oplus n_1} \oplus \cdots \oplus \Phi(F_k \otimes H^m)^{\oplus n_k}. $$
	For each $i$, the condition $(\ddagger)$ guarantees that the last component of the Mukai vector of $F_i \otimes H^m$ is at least $2$, which implies that the rank of $\Phi(F_i \otimes H^m)$ is at least $2$, hence $\Phi(F \otimes H^m)$ satisfies the condition (ii) in Proposition \ref{prop:known}. Moreover each $\Phi(F_i)$ is $\mu_H$-stable. By Proposition \ref{prop:known} and Remarks \ref{rmk:nongen}, \ref{rmk:nomust}, the DG algebra $\RHom^\hdot(\Phi(F \otimes H^m), \Phi(F \otimes H^m))$ is formal, hence $\RHom^\hdot(F,F)$ is also formal by Proposition \ref{prop:intro}.
	
	The case of $(r, dH+D)=(0,0)$ is reduced to Proposition \ref{prop:known} (i) by applying \eqref{eqn:FMtrans}. \qed

\subsection{Formality  for derived objects}\label{sec:derived}
Let $\Stab^\dagger(X)$ be the connected component of the space of stability conditions on $X$ which contains the geometric ones; see \cite[Definition 11.4]{Bri08}. The following was communicated to us by K. Yoshioka and A. Bayer. Since it is well-known to experts, at the advice of the referees we leave out the details. The reader can find however a full proof in an earlier version of this article \cite{BZv3}.

\begin{proposition}
	\label{prop:derived}
	Let $X$ be a  projective K3 surface, $\vv$ a Mukai vector, and $\sigma \in \Stab^\dagger(X)$  generic with respect to $\vv$. Then there exists an autoequivalence
	$$ \Theta: D^b(\Coh(X)) \longrightarrow D^b(\Coh(X)) $$
	which induces an isomorphism $ \cM_{X,\sigma}(\vv) \cong \cM_{X,\omega}(\uu) $ preserving S-equivalence classes, between the moduli space  of $\sigma$-semistable objects of class $\vv$, and the moduli space  of $\omega$-semistable sheaves of class $\uu$ for some generic ample class $\omega$ on $X$. \end{proposition}

\begin{proof}
For $\vv^2>0$, this is essentially \cite[Lemma 7.3]{BM14a}, generalized to the current form by an idea of K. Yoshioka; see \cite[Remark 3.15]{MZ16}.

For $\vv^2 \leqslant 0$, the idea of the proof is due to A. Bayer. By \cite[Lemma 7.1, Lemma 7.2 (b)]{BM14a}, we can assume $\vv$ is primitive. We can also assume the leading component $\vv_0>0$. Indeed, if $\vv_0 < 0$, we can apply the shift functor $[1]$, so that $\vv$ gets replaced by $-\vv$. If $\vv_0 = 0$, after tensoring with a line bundle if necessary, we can assume $\vv_2 \neq 0$. Then one applies \eqref{eqn:FMtrans} to obtain $\widehat{\vv}$ whose leading component is non-zero.

The rest of the proof makes use of the wall-crossing technique. By \cite[\S 9]{Bri08},  $\Stab^\dagger(X)$ admits a wall and chamber structure. There is one chamber which contains $\sigma$ as an interior point, and another ``Gieseker chamber'' in which we can pick a stability condition $\tau$, such that the $\tau$-stability for class $\vv$ is the same as the Gieseker $\beta$-twisted $\omega$-stability for some generic $\beta \in \NS(X)_\QQ$ and $\omega \in \Amp(X)_\QQ$; see \cite[\S 14]{Bri08}. The assumptions $\vv_0 > 0$ and $\omega$ being generic imply further that the $\beta$-twisted $\omega$-stability for class $\vv$ is the same as the untwisted $\omega$-stability by an argument similar to \cite[Lemma 1.1]{Yos01}. We can move $\sigma$ to $\tau$ in $\Stab^\dagger(X)$ along a path that never meets two walls simultaneously. For each wall-crossing, we can construct an explicit  autoequivalence of $D^b(\Coh(X))$ which induces an isomorphism of the moduli spaces of stable objects with respect to generic stability conditions in the neighboring chambers separated by the wall. The idea of its explicit construction in the case of $\vv^2<0$ is essentially contained in \cite[Proposition 6.8]{BM14b}, and in the case of $\vv^2=0$ it is a combination of \cite[Lemma 7.2(a)]{BM14a} and the twisted K3 surface version of \cite[Theorem 12.1]{Bri08} which can be found in \cite[Section 3.1]{HMS08}, see \cite[Remark 6.4]{BM14a}.
\end{proof}

\noindent{\bf{Proof of Theorem \ref{cor:intro}.}}
	By Proposition \ref{prop:derived}, $\Theta(F)$ is an $L$-polystable coherent sheaf on $X$. By Theorem \ref{thm:intro}, the DG algebra $\RHom^\hdot(\Theta(F),\Theta(F))$ is formal, which implies that $\RHom^\hdot(F,F)$ is formal by Proposition \ref{prop:intro}. \qed

\end{document}